\documentclass{article}

\usepackage{amsfonts}
\usepackage{amsmath}
\usepackage{amsthm}
\newtheorem{theorem}{Theorem}
\newtheorem{q}[theorem]{Question}
\newtheorem{conj}[theorem]{Conjecture}
\newtheorem{lemma}[theorem]{Lemma}

\def\I{\mathcal{I}}

\title{Tur\'an and Ramsey Properties of Subcube Intersection Graphs} 
\author{J.~Robert Johnson\thanks{School of Mathematical Sciences, Queen Mary University of London, London E1 4NS, UK
(\tt{r.johnson@qmul.ac.uk})} \and Klas Markstr\"om\thanks{Department of Mathematics and Mathematical Statistics, Ume\aa{\ } University,
Ume\aa \ 901 87, Sweden (\tt{klas.markstrom@math.umu.se})}} 

\begin{document}
\maketitle

\begin{abstract}
The discrete cube $\{0,1\}^d$ is a fundamental combinatorial structure. A subcube of $\{0,1\}^d$ is a subset of $2^k$ of its points formed by fixing $k$ coordinates and allowing the remaining $d-k$ to vary freely. The subcube structure of the discrete cube is surprisingly complicated and there are many open questions relating to it. 

This paper is concerned with patterns of intersections among subcubes of the discrete cube. Two sample questions along these lines are as follows: given a family of subcubes in which no $r+1$ of them have non-empty intersection, how many pairwise intersections can we have? How many subcubes can we have if among them there are no $k$ which have non-empty intersection and no $l$ which are pairwise disjoint?

These questions are naturally expressed using intersection graphs. The intersection graph of a family of sets has one vertex for each set in the family with two vertices being adjacent if the corresponding subsets intersect. Let $\I(n,d)$ be the set of all $n$ vertex graphs which can be represented as the intersection graphs of subcubes in $\{0,1\}^d$. With this notation our first question above asks for the largest number of edges in a $K_{r+1}$-free graph in $\I(n,d)$. As such it is a Tur\'an type problem. We answer this question asymptotically for some ranges of $r$ and $d$. More precisely we show that if $(k+1)2^{\lfloor\frac{d}{k+1}\rfloor}<n\leq k2^{\lfloor\frac{d}{k}\rfloor}$ for some integer $k\geq 2$ then the maximum edge density is $\left(1-\frac{1}{k}-o(1)\right)$ provided that $n$ is not too close to the lower limit of the range.

The second question can be thought of as a Ramsey type problem. The maximum such $n$ can be defined in the same way as the usual Ramsey number but only considering graphs which are in $\I(n,d)$.  We give bounds for this maximum $n$ mainly concentrating on the case that $l$ is fixed and make some comparisons with the usual Ramsey number.

Tur\'an and Ramsey type problems are at the heart of extremal combinatorics and so these problems are mathematically natural. However, a second motivation is a connection with some questions in social choice theory arising from a simple model of agreement in a society. Specifically, if we have to make a binary choice on each of $n$ separate issues then it is reasonable to assume that the set of choices which are acceptable to an individual will be represented by a subcube. Consequently, the pattern of intersections within a family of subcubes will have implications for the level of agreement within a society. A further motivation is the fact that our subcube intersection graphs can be thought of as a discrete analogue of $d$-box graphs. The idea of $d$-box graphs has turned out to be a fruitful concept for applications in areas such as ecological and social networks. It may be hoped that our subcube intersection graphs will find similar application.

We pose a number of questions and conjectures relating directly to the Tur\'an and Ramsey problems as well as raising some further directions for study of subcube intersection graphs. 

\end{abstract}

\section{Introduction}

Suppose that $A_1,\dots A_n$ are subsets of some ground set $X$. The associated intersection graph is formed by taking one vertex for each $A_i$ with two vertices being adjacent if the corresponding subsets intersect. Intersection graphs are well-studied objects (see the book \cite{McMc} and references therein). In some cases $X$ is an arbitrary finite set but more usually $X$ has some structure and we allow only certain subsets. A simple example is the family of interval graphs where we take $X=\mathbb{R}$ with the subsets $A_i$ being intervals. More generally $d$-box graphs \cite{R} have $X=\mathbb{R}^d$ with the subsets being axis-parallel boxes (that is products of intervals). 

In this paper we are concerned with a natural discrete analogue of $d$-box graphs. Specifically we will take our ground set to be $\{0,1\}^d$ (the vertices of the discrete cube) with our subsets being subcubes (a subcube is a set of points of the discrete cube formed by fixing $k$ coordinates and allowing the remaining $d-k$ to vary freely). One feature of this situation is that the Helly property holds: if we have a family of pairwise intersecting subcubes then there is a point of the ground set contained in all of them. Thus the intersection graph tells us everything about the intersection structure of the set system. 

We mention in passing that the problem of representing a graph as the intersection graph of discrete subcubes has been considered in a slightly different guise \cite{JK}. Specifically the \emph{biclique cover number} of a graph is the minimum $d$ for which the edges of $G$ can be covered by $d$ complete bipartite subgraphs of $G$. It is easy to see that this is also the smallest $d$ for which the complement of $G$ can be represented as the intersection graph of subcubes in $\{0,1\}^d$. 

Our interest is rather different. Instead of considering how to represent a given graph as the intersection graph of subcubes we will consider the set of all graphs which can be represented as the intersection graph of $n$ subcubes in $\{0,1\}^d$. In particular we will be interested in Tur\'an and Ramsey properties of this set of graphs. We will pose precise questions shortly but roughly speaking Tur\'an type questions ask how many edges do we need to guarantee a particular subgraph while Ramsey type questions ask how many vertices do we need to guarantee either a large complete subgraph or a large empty subgraph. For background to Tur\'an and Ramsey problems in graph theory as well as basic graph theoretic terminology and notation see \cite{MGT}. These two types of extremal problem lie at the heart of extremal combinatorics so these questions are extremely natural mathematically. In addition intersection graphs have been used in a range of applications in many areas from biology to computing (\cite{McMc} provides brief descriptions of a number of applications and further references). Our subcube intersection graphs and the extremal questions we study arise very naturally from a simple model in social choice theory which we describe next.

One motivation for the study of certain intersection graphs (as described in \cite{BNSTW}) comes from a model for agreement in a society based on approval voting. Here we have a set $X$ of all possible preferences (the ``political spectrum'') and $n$ individuals each of which finds some subset of $X$ to be acceptable (we refer to this subset as the individual's ``approval set''). Usually we assume that $X$ has some structure and the possible approval sets reflect this. For instance we could take $X=\mathbb{R}$ and insist that all approval sets are intervals. The connections between this model and various geometric intersection theorems were explored by Berg, Norine, Su, Thomas and Wollan \cite{BNSTW}. Suppose that we take a set of individuals $[d]=\{1,2,\dots,d\}$ and let individual $i$ have approval set $A_i$. Let $G$ be the intersection graph on these approval sets. If $A_i\cap A_j\not=\emptyset$ then there is a point of the political spectrum which is acceptable to both $i$ and $j$; in this case we will say that $i$ and $j$ \emph{agree}. Now the number of edges in $G$ is the number of pairs of individuals who agree. The clique number of $G$ is the size of the largest set of individuals who pairwise agree (if the family of possible approval sets has the Helly property then this is the same as the largest set of individuals who collectively find some point of $X$ acceptable). The independence number of $G$ is the largest set of individuals no two of whom agree. These are all natural things to consider in the agreement model and Berg et al. \cite{BNSTW} raise some interesting questions on them. They are mainly concerned with the case $X=\mathbb{R}^d$ with either all approval sets being convex sets (here the Helly property does not hold) or all approval sets being axis-parallel boxes (here the Helly property does hold). As an example of their results they prove that if $G$ is an interval graph on $n$ vertices and within every set of $m$ vertices we can find a copy of $K_k$ then the clique number of $G$ is at least $\frac{k-1}{m-1}n$ (they refer to a society in which the hypotheses of this result hold as being $(k,m)$-agreeable). In agreement language this says that if among any $m$ individuals we can find $k$ who agree then we can find a proportion $\frac{k-1}{m-1}$ of the whole society who agree. 

Suppose that we wish to describe preferences and agreement over $d$ binary (yes or no) issues. The natural political spectrum to take here is $X=\{0,1\}^d$. An individual may have a view in either direction on some of the issues and have no preference on the remaining ones. In other words the approval sets will be subcubes of $\{0,1\}^d$. This motivates the study of the intersection graphs of subcubes.

Having given some background informally let us now introduce some notation. Let $C_d=\{0,1,*\}^d$. We represent a subcube as a vector $u\in C_d$ where $u$ represents the subcube of $\{0,1\}^d$ consisting of all points $x\in\{0,1\}^d$ for which $u_i=x_i$ whenever $u_i\not=*$. For $u\in C_d$ we will write $F(u)=\{i:u_i\not=*\}$ for the set of fixed coordinates of $u$. The dimension of $u$ is $d-|F(u)|$ and the codimension of $u$ is $|F(u)|$. With this notation an intersection graph of subcubes has as vertex set some multiset from $C_d$ with $x,y$ adjacent if $x_i=y_i$ for all $u\in F(x)\cap F(y)$. Let $\I(n,d)$ be the set of all graphs on $n$ vertices which can be realised as intersection graphs of subcubes of $\{0,1\}^d$.  

In this paper we will mainly be concerned with two questions each of which is both mathematically natural and relevant to the agreement model. 

The first is a Tur\'an type question. Suppose that $G\in\I(n,d)$ is a $K_{r+1}$-free graph, how many edges can $G$ have? In agreement terms this will give a lower bound on the largest set of individuals we can guarantee to find all of whom agree on some point of $X$ in terms of the number of pairs of individuals who agree. A result of this type for intersection graphs of boxes in $\mathbb{R}^d$ was proved by Eckhoff \cite{Eckhoff}. For convex sets in $\mathbb{R}^d$ the fractional Helly theorem of Kalai \cite{K} has some similarity although it considers $d+1$-wise intersections.

The second is a Ramsey type question. Suppose that $G\in\I(n,d)$ contains no independent set of size $l$, how small can the clique number of $G$ be? In agreement terms this asks for a lower bound on the largest set of individuals we can guarantee to find all of whom agree on some point of $X$ under the assumption that we have no $l$ individuals who pairwise disagree (compare this with the condition in the result on $(k,m)$-agreeable societies from \cite{BNSTW} above).

There are numerous open questions and directions for further study involving subcube intersection graphs. These include extending the extremal results we prove here as well as other issues such as random subcube intersection graphs. In the final section we collect some of these questions. This is in addition to a number of conjectures and problems which are raised in the main body of the paper.

\section{Tur\'an problems}

A general Tur\'an type problem asks for the largest size a structure can be without containing a particular forbidden substructure. Tur\'an's theorem itself states that the maximum number of edges in a $K_{r+1}$-free graph is attained by a complete $r$-partite graph with parts as equal as possible (the so called Tur\'an graph). There are numerous extensions, for instance  to other forbidden graphs and more generally to other combinatorial structures such as hypergraphs. Chapter 4 of \cite{MGT} provides a good introduction to Tur\'an type problems in graphs

\begin{q}
What is the maximum number of edges in a $K_{r+1}$-free graph in $\I(n,d)$?
\end{q}

It turns out to be helpful to consider different ranges of $n$ separately as the behaviour of the extremal number of edges varies. Generally we will think of $r$ as being fixed, $d$ being large, and $n$ being in some range depending on $d$ (and $r$). Note that if $n>r2^d$ then some point of $\{0,1\}^d$ is contained in at least $r+1$ of our subcubes (by the pigeon-hole principle) and so there are no $K_{r+1}$-free graphs in $\I(n,d)$ for such $n$. Below this the maximum number of edges we can have is increasing in $n$.

\begin{lemma}\label{increase:lem}
Suppose that $n<n'\leq r2^d$. If $G\in\I(n,d)$ is $K_{r+1}$-free then we can find a $K_{r+1}$-free graph $G'\in\I(n',d)$ with $e(G')\geq e(G)$.
\end{lemma}

\begin{proof}
It suffices to consider the case $n'=n+1$. If all subcubes corresponding to vertices in $V(G)$ are singletons (that is they have dimension 0) then some subcube occurs at most $r-1$ times. Adding an extra copy of this subcube to the vertex set gives a suitable $G'$. Otherwise we have some subcube $v\in V(G)$ for which $v_i=*$ for some $i$. Delete this vertex and adding two new vertices $v^0$ and $v^1$ with $v^0_i=0$, $v^1_i=1$ and $v^0_j=v^1_j=v_j$ for all $j\not=i$ gives a suitable $G'$.
\end{proof}

The next two theorems give two simple upper bounds on the number of edges we can have, each of which can be achieved in a certain range.

Let $T_r(n)$ be the $r$-partite Tur\'an graph on $n$ vertices and $t_r(n)=e(T_r(n))$.

\begin{theorem}\label{turan:thm}
If $G\in\I(n,d)$ is $K_{r+1}$-free then $e(G)\leq t_{r}(n)$. This bound is best possible for $n\leq r2^{\lfloor d/r\rfloor}$.
\end{theorem}

\begin{proof} The bound is simply Tur\'an's theorem for any $K_{r+1}$-free graph.

To achieve this let $\lfloor d/r\rfloor=t$ and let $P_1,\dots, P_r$ be pairwise disjoint $t$-sets in $[d]$. Let $V$ be the collection of all subcubes $u\in C_d$ with $F(u)=P_i$ for some $i=1,\dots,r$. Any two cubes in $V$ with the same fixed set do not intersect while any two with different fixed sets do intersect so this graph is complete $r$-partite with $2^t$ vertices in each class. Taking a suitable subgraph shows that $T_r(n)\in\I(n,d)$ for any $n\leq r2^t$.
\end{proof}

\begin{theorem}\label{absolute:thm}
If $G\in\I(n,d)$ is $K_{r+1}$-free then $e(G)\leq \binom{r}{2} 2^d$. This bound is best possible for $n\geq r 2^{d-\lfloor d/r\rfloor}$.
\end{theorem}

\begin{proof}
Let $V$ be the vertex set of a $K_{r+1}$-free graph $G$ in $\I(n,d)$. Each point $x\in \{0,1\}^d$ lies in at most $r$ subcubes in $V$. It follows that each such $x$ is involved in at most $\binom{r}{2}$ intersections between pairs of subcubes in $V$. Hence $e(G)\leq\binom{r}{2} 2^d$.

As before let $\lfloor d/r\rfloor=t$ and let $P_1,\dots, P_r$ be pairwise disjoint $t$-sets in $[d]$. Let $V$ be the collection of subcubes $u$ with $F(u)=[d]\setminus P_i$ for some $i=1,\dots,r$. Now any two subcubes $x,y$ in $V$ which intersect have $F(x)\cup F(y)=[d]$. It follows that $x$ and $y$ intersect in only a single point of $\{0,1\}^d$. Further every point of $\{0,1\}^d$ is contained in exactly $r$ subcubes in $V$. Hence the bound is attained when $n=r 2^{d-\lfloor d/r\rfloor}$. For larger $n$ we apply Lemma \ref{increase:lem}.
\end{proof}

Note that in both the constructions described above, if $r$ does not divide $d$ we can modify the $P_1,\dots,P_r$ so that they form a partition of $[d]$ by letting some of the $P_i$ have size $\lceil d/r\rceil$. This gives a slightly larger range of $n$ for which the bounds can be attained.  

The simplest large $K_{r+1}$-free graph is a complete $r$-partite graph. Following the construction in the proof of Lemma \ref{turan:thm}, a natural way of constructing a complete $r$-partite graph in $\I(n,d)$ is to fix a partition $[d]=P_1\cup\dots\cup P_r$ and to take all elements of $C_d$ whose fixed set is some $P_i$. If $|P_i|=d_i$ then this graph has $\sum_{i=1}^r 2^{d_i}$ vertices and $\binom{n}{2}-\sum_{i=1}^r\binom{2^{d_i}}{2}$ edges. For certain values of $n$ we may have to take a subgraph of the graph constructed in this way.
We will describe this type of construction as an $r$-partite construction based on a partition of $[d]$ and conjecture that for small $n$ (made precise in the statement below) the extremal graph for our problem has this form.

\begin{conj}\label{smalln:con}
For any $n\leq 2\times2^{d/2}+(r-2)$ the maximum number of edges in a $K_{r+1}$-free graph in $\I(n,d)$ is attained by a graph $G$ which is a subgraph of a graph $H$ of the form  
\[
V(H)=\{x\in C_d: F(x)=P_i\}
\]
where $P_1,\dots,P_r$ is a partition of $[d]$.
\end{conj}

Finding the maximum number of edges a graph constructed in this way can have (allowing also for subgraphs of it) is equivalent to solving the following optimization problem.

Given $n,r,d$ find $n_1,\dots,n_r,d_1,\dots,d_r\geq 0$ to minimize 
\[
\sum_{i=1}^r\binom{n_i}{2}
\]
subject to
\begin{align*}
\sum_{i=1}^{r} d_i&=d\\
\sum_{i=1}^{r} n_i&=n\\
n_i&\leq 2^{d_i}.
\end{align*}

Given $d_i$ it is clear that we should choose the $n_i$ to be as equal as possible subject to the constraints. However, when we are free to choose the $d_i$ it does not seem easy to give an exact solution to this optimization problem. Despite this we can use the construction to give lower bounds which in some cases have asymptotically optimal edge density. This theorem applies to the range $r 2^{d/r}<n\leq 2\times 2^{d/2}$. For smaller $n$ Lemma \ref{turan:thm} applies while larger $n$ will be considered later.

\begin{theorem}\label{smalln:thm}
If $n\leq k2^{\lfloor\frac{d}{k}\rfloor}$ with $2\leq k\leq r$ an integer then there is a $K_{r+1}$-free graph in $\I(n,d)$ with $t_k(n)$ edges.

Moreover, for all $\epsilon>0$ there exists $\delta>0$ such that if $k$ is as above and $\frac{1}{n}2^{\frac{d}{k+1}}<\delta$ then the maximum number of edges in a $K_{r+1}$-free graph in $\I(n,d)$ is $(1-\frac{1}{k}+\epsilon)\binom{n}{2}$.
\end{theorem}

The first part of the theorem gives that for $n$ in the range
\[
(k+1)2^{\lfloor\frac{d}{k+1}\rfloor}<n\leq k2^{\lfloor\frac{d}{k}\rfloor}
\]
we have an upper bound of $(1-\frac{1}{k}-o(1))\binom{n}{2}$. The second part of the theorem shows that this upper bound is asymptotically tight for large $n$ (for which we need large $d$ also) provided that $n$ is not too close to $(k+1)2^{\lfloor\frac{d}{k+1}\rfloor}$. Since the function $x2^{d/x}$ is decreasing for $0<x<d$ it is almost the case that given $n$ we have a unique $k$ with $n$ in a range of this form. However because of the rounding in the exponent this is not quite the case. For some values of $d,k$ we have $k2^{\lfloor\frac{d}{k}\rfloor}<(k+1)2^{\lfloor\frac{d}{k+1}\rfloor}$ from which it follows that some ranges of this form will be empty while some will overlap. Nevertheless the statement of the Theorem accounts for this. To get the best upper bound we should consider the largest $k$ for which $n\leq k2^{\lfloor\frac{d}{k}\rfloor}$.

We will need the following result of Erd\H os and Simonovits \cite{ES} on the minimum number of copies of $K_t$ in a graph of given edge density.

\begin{theorem}[(Supersaturation Theorem)]\label{supersat:thm}
For all $c>0$ there exists $c'>0$ such that every graph $G$ with $n$ vertices, $(1-\frac{1}{k}+c)\frac{n^2}{2}$ edges contains at least $c'n^{k+1}$ copies of $K_{k+1}$.
\end{theorem}

In the proof below and elsewhere $\log$ denotes the base 2 logarithm.

\begin{proof}[of Theorem~\ref{smalln:thm}]
Take $t=\lfloor\log\frac{n}{k}\rfloor$. By the conditions of the theorem we have that $n\leq k2^{\lfloor d/k\rfloor}$ and so $tk\leq d$. It follows that we can take $k$ pairwise disjoint $t$-sets $P_1,\dots, P_k$ in $[d]$. Let $V$ be the collection of all subcubes $u$ with $F(u)=P_i$ for some $i=1,\dots,k$.
The corresponding intersection graph is the Tur\'an graph $T_k(k2^{\lfloor d/k\rfloor})$. By taking a suitable subgraph of this we obtain $T_k(n)$.

For the upper bound let $V$ be the vertex set of a $K_{r+1}$-free graph $G$ in $\I(n,d)$ and suppose that $e(G)>\left(1-\frac{1}{k}+\epsilon\right)\frac{n^2}{2}$. We have that each $x\in\{0,1\}^d$ is in at most $r$ subcubes in $V$. Further, by the Helly property, for each copy of $K_{k+1}$ there is some point of $\{0,1\}^d$ common to all the $k+1$ subcubes forming the $K_{k+1}$. Consequently the graph $G$ contains at most $\binom{r}{k+1}2^d$ copies of $K_{k+1}$ (the case $k=1$ of this is contained in Theorem \ref{absolute:thm}). 

The edge density of $G$ means that we can apply Theorem \ref{supersat:thm} to get that that the number of copies of $K_{k+1}$
in $G$ is at least $\epsilon'n^{k+1}$ for some $\epsilon'>0$. Hence $\epsilon'n^{k+1}\leq\binom{r}{k+1}2^d$ and so $\frac{1}{n}2^{\frac{d}{k+1}}\geq\frac{\epsilon'}{\binom{k}{r+1}}$. The result follows if we set $\delta=\frac{\epsilon'}{\binom{k}{r+1}}$.
\end{proof}

The above construction is clearly not exactly optimal. For instance if $k<r$ we may replace one subcube in $V$ with $(*,\dots,*)$ giving a graph with more edges (effectively we are replacing our balanced $k$-partite graph with a $(k+1)$-partite graph with one class containing a single vertex). If in addition $kt<d$ then we could replace $2^{d-kt}$ vertices with a $(k+1)$st class of $2^{d-kt}$ subcubes each of whose fixed set is $[d]\setminus(P_1\cup\dots\cup P_k)$. 

We suspect that when $n$ is between $(k+1) 2^{d/(k+1)}$ and $k2^{d/k}$ the extremal graph is an $r$-partite graph based on a partition of $[d]$ in which $k$ parts are large (around $\log \frac{n}{k}$) and $k-r$ are small. It is messy to calculate the exact number of edges but this informal description gives the idea. If $\frac{\log n}{d}$ is bounded away from $k+1$ then the small parts will have size $o(\log n)$ and we have a good asymptotic bound. However if $\frac{\log n}{d}$ is slightly larger than $k+1$, the small parts will be more significant and we do not have a good upper bound. Using more precise results on the minimum number of triangles and $K_4$s in a graph we can show more in the cases $r=3,4$. 

We will need the following result on the clique density problem due to Razborov (in the case $r=3$) \cite{Raz} and Nikiforov
\cite{N}. 

\begin{theorem}\label{RN:thm}
For $r=3,4$ if $G$ is a graph on $n$ vertices with $\alpha\binom{n}{2}$ edges where $\alpha\in\left[1-\frac{1}{t-1},1-\frac{1}{t}\right]$ for integer $t$ then the number of $K_r$ in $G$ is asymptotically at most the number of $K_r$ in the $t$-partite graph of edges density $\alpha$ with $t-1$ roughly equal parts and one smaller part.
\end{theorem}

Using this we get a better asymptotic result for $r=3,4$ which improves on our earlier upper bound when $n$ is a little larger than $r\times 2^{d/r}$.

\begin{theorem}\label{smalln2:thm}
For $r=3,4$ if $G$ is a $K_{r+1}$-free graph in $\I(n,d)$ where $n=(r-1)\times2^x+2^{d-(r-1)x}$ for some $d/r\leq x\leq d/(r-1)$ then the maximum number of edges in $G$ is attained asymptotically by the complete $r$-partite graph with $(r-1)$ parts of size $2^x$ and one part of size $2^{d-(r-1)x}$.
\end{theorem}

\begin{proof}
Note first that the required graph is indeed in $\I(n,d)$; it is the $r$-partite graph based on a partition of $[d]$ into $(r-1)$ parts of size $x$ and one of size $d-(r-1)x$.

This graph contains $2^d$ copies of $K_r$ which we know is the maximum number possible in a $K_{r+1}$ free graph in $\I(n,d)$. Now by Theorem \ref{RN:thm} any graph on $n$ vertices with a strictly larger edge density must have a strictly larger density of copies of $K_r$. It follows that no such graph can be in $\I(n,d)$ and so this graph is asymptotically best possible.
\end{proof}

An answer to the clique density problem valid for all $r$ would allow us to extend the previous result to all $r$. However even this would only work for specific values of $n$, all of them in the range a little larger than $r2^{d/r}$. 

When $n=2\times2^{d/2}$ the construction of Theorem \ref{smalln:thm} gives a complete bipartite graph which has edge density $1/2$. By the second part of the same theorem this is asymptotically optimal. To get slightly more edges for $n$ up to $2\times2^{d/2}+(r-2)$ we could take $G$ to be a subgraph of the $r$-partite graph based on a partition of $[d]$ into 2 $d/2$-sets and $(r-2)$ copies of $\emptyset$. This explains the range of $n$ for which we make Conjecture \ref{smalln:con}.

Notice that as soon as $n$ becomes asymptotically larger than $2^{d/2}$ the upper bound of Lemma \ref{absolute:thm} implies that we cannot have a positive proportion of edges in a $K_{r+1}$-free graph. This suggests that determining the extremal number of edges is likely to be more delicate than in the smaller $n$ region and indeed we have no better upper bound than $\binom{r}{2}2^d$. In this large $n$ range complete $r$-partite constructions based on partitions of $[d]$ do not give good lower bounds. Instead we consider constructions more like that used in the proof of Lemma \ref{absolute:thm}.

It is possible to construct a $K_{r+1}$-free graph in $\I(n,d)$ in a similar way to the constructions based on partitions of $[d]$ but based on \emph{any} subsets of $[d]$. Specifically, for some $k\leq r$ let $R_1,\dots,R_k$ be subsets of $[d]$. Now take all elements of $C_d$ whose fixed set is some $R_i$. This graph will be $k$-partite and so $K_{r+1}$-free. However, if the $R_i$ are not pairwise disjoint then it will not be complete $k$-partite. The number of edges between the classes corresponding to $R_i$ and $R_j$ is $2^{|R_i\cup R_j|}$ (there are $2^{|R_i|}$ vertices in the $i$th class and each is adjacent to $2^{|R_j\setminus R_i|}$ vertices in the $j$th class). If $k<r$ it will always be better to add a new class based on the empty set (this will consist of the singleton vertex corresponding to the subcube $(*,\dots,*)$ which will be joined to every other vertex). Since this increases the number of vertices by one we should then delete any vertex. So provided we are allowed to consider subgraphs of graphs constructed like this we may as well take $k=r$ (as we did in Conjecture \ref{smalln:con}).

\begin{theorem}\label{largen:thm}
If $n\geq k2^{d-\lceil\frac{d}{k}\rceil}$ with $2\leq k\leq r$ an integer then there is a $K_{r+1}$-free graph in $\I(n,d)$ with $\binom{k}{2}2^d$ edges.

\end{theorem}

Note that the although the function $x2^{d-d/x}$ is increasing in $x$ the rounding once again makes things a little complicated. For the best bound we should take the largest $k$ for which the inequality $n\geq k2^{d-\lceil\frac{d}{k}\rceil}$ holds.

\begin{proof}
Take $t=\lfloor d-\log\frac{n}{k}\rfloor$. By the conditions of the theorem we have that $tk\leq d$. It follows that we can take $k$ pairwise disjoint $t$-sets $P_1,\dots, P_k$ in $[d]$. Now set $R_i=[d]\setminus P_i$ and let $V$ be the collection of all subcubes $u$ with $F(u)=R_i$ for some $i=1,\dots,k$. This graph is of the form described above and so it is $K_{r+1}$-free and has $\binom{k}{2}2^d$ edges since $|R_i\cup R_j|=d$ for all $i\not=j$.
\end{proof}

As we remarked above this bound can be improved slightly by introducing $r-k$ new classes with $P_i=\emptyset$ and deleting $r-k$ vertices from the original graph. However this will only give a small improvement. We conjecture that a construction of this form is optimal.

\begin{conj}\label{largen:con}
For any $n\geq 2\times2^{d/2}+(r-2)$ the maximum number of edges in a $K_{r+1}$-free graph in $\I(n,d)$ is attained by a graph $G$ which is a subgraph of a graph $H$ of the form  
\[
V(H)=\{x\in C_d: F(x)=R_i\}
\]
where if $P_i=[d]\setminus R_i$ then $P_1,\dots,P_k$ is a partition of $[d]$ for some $k$ and $P_{k+1},\dots,P_r=[d]$.
\end{conj}

\section{Tur\'an problems with arbitrary finite $X$}

As we mentioned earlier, in most work on intersection graphs there is some underlying structure on the ground set. However some problems are also natural to consider for arbitrary ground sets. One example is an early result of Erd\H os, Goodman and P\'osa \cite{EGP} that any graph on $n$ vertices is an intersection graph of subsets of a $\lfloor\frac{n^2}{4}\rfloor$-set.
In this section we digress briefly to consider our Tur\'an type problem for an arbitrary ground set. Specifically we consider the situation where rather than subcubes of $\{0,1\}^d$ we allow arbitrary subsets of a finite set. This results in some natural questions in extremal set theory which we answer almost completely. These results may be of independent interest as well as being a contast to the subcube results. 

Suppose that $X$ is an $m$-set and $A_1,\dots,A_n$ are any subsets of $X$. The Tur\'an type problem now splits into two separate questions both of which we address here. We can ask either for the maximum number of edges in the corresponding intersection graph under the assumption that it is $K_{r+1}$-free or under the assumption that no element of $X$ is contained in more than $r$ of the $A_i$. For subcubes these are equivalent because of the Helly property but for arbitrary subsets they are not.

\begin{theorem}\label{general1:thm}
Let $G$ be an $n$-vertex intersection graph with ground set $X$ of size $m$. If $G$ is $K_{r+1}$-free then $n\leq rm$ and  
\[
e(G)\leq\min\{t_r(n),\binom{r}{2}m\}.
\]
Moreover, if $\sqrt{m}$ is a prime power and $r\leq\sqrt{m}+1$ then this bound is sharp for all $n$.
\end{theorem}

\begin{proof}
If will suffice to find sets $A_1,\dots, A_n$ whose intersection graph $G$ has number of edges equal to both $t_r(n)$ and $\binom{r}{2}m$. For smaller $n$ we can take subgraphs of this. For larger $n$ we can repeatedly replace some non-singleton $A_i$ with a partition of it into two non-empty sets. This operation does not decrease the number of edges and so the graph formed still has $\binom{r}{2}m$ edges and so attains the upper bound. Note that eventually we have $r$ copies of each singleton in $X$ and this is plainly the maximum number of vertices we can have without a $K_{r+1}$.
 
Since $\sqrt m=q$ is a prime power and $r\leq q+1$ we can find $r-2$ mutually orthogonal latin squares of order $q$ (see section 5.2 of \cite{DK} for instance). Let $X=\{(i,j):1\leq i,j\leq q\}$. Take $R_x=\{(i,j):i=x\}$, $C_y=\{(i,j):j=y\}$ and 
\[
S_x^k=\{(i,j): \text{ the $k$th Latin square has symbol $x$ in position }(i,j)\}.
\]
Now the intersection graph of the sets 
\[
R_1,\dots,R_q,C_1,\dots,C_q,S_1^1,\dots,S_q^1,\dots,S_1^{r-2},\dots,S_q^{r-2}
\]
is the Tur\'an graph $T_r(rq)$. This graph has $\binom{r}{2}q^2=\binom{r}{2}m$ edges and so both bounds are attained.
\end{proof}

\begin{theorem}\label{general2:thm}
Let $A_1,\dots,A_n$ be subsets of an $m$-set $X$ and let $G$ be the associated intersection graph. If no element of $X$ is contained in more than $r$ of the $A_i$ then $n\leq rm$ and:
\[
e(G)\leq\min\{\binom{n}{2},\binom{r}{2}m\}.
\]
Moreover, given $\epsilon>0$ and $m\geq m_0(r,\epsilon)$ then this bound is sharp for all $n<r\sqrt{m}(1-\epsilon)$ and all $rm\geq n>r\sqrt{m}(1+\epsilon)$.
\end{theorem}

\begin{proof}
Plainly $e(G)\leq\binom{n}{2}$. The argument of Lemma \ref{absolute:thm} applies in this case also to show that $e(G)\leq\binom{r}{2}m$.

For the lower bound if $n<r\sqrt{m}(1-\epsilon)$ take a collection $S_1,\dots,S_m$ of $r$-subsets of $[n]$ with the property that every pair in $[n]^{(2)}$ is contained in at least one $S_i$. We can do this if $n=r\sqrt{m}(1-\epsilon)$ and $m$ is sufficiently large by R\"odl's proof of the Erd\H os-Hanani conjecture \cite{Rodl}. It is now easy to adapt such a collection of sets for smaller $n$. Now let $A_i=\{a: i\in S_a\}$ and consider the family of sets $A_1,\dots,A_n\subseteq[m]$. Since $|S_a|=r$ for all $a$, every element of $[m]$ is contained in only $r$ of the $A_i$. Now given $i,j\in[n]$ we have that there is some $x\in[m]$ with $i,j\in S_x$ and so $x\in A_i\cap A_j$. In particular $A_i\cap A_j\not=\emptyset$ for all $i,j$ and so $e(G)=\binom{n}{2}$ which is $\min\{\binom{n}{2},\binom{r}{2}m\}$ for this range of $n$.

If $n>r\sqrt{m}(1+\epsilon)$ we can similarly take a collection $S_1,\dots,S_m$ of $r$-subsets of $[n]$ with the property that every pair in $[n]^{(2)}$ is contained in at most one $S_i$. Again let $A_i=\{a: i\in S_a\}$ and consider the family of sets $A_1,\dots,A_n\subseteq[m]$. Since $|S_a|=r$ for all $a$, every element of $[m]$ is contained in only $r$ of the $A_i$. Since no $i,j\in[n]$ are contained in more than one of the $A_i$ is follows that $|A_i\cap A_j|\leq 1$ for all $i,j$ and so $e(G)=\binom{r}{2}m$ which is $\min\{\binom{n}{2},\binom{r}{2}m\}$ for this range of $n$.
\end{proof}

Note that this argument shows that if $n$ is such that there is a collection of $r$-sets containing every pair in $[n]^{(2)}$ exactly once then the bound of Theorem \ref{general2:thm} is sharp when $m=\frac{\binom{n}{2}}{\binom{r}{2}}$.

\section{Ramsey problems}

The theme of Ramsey theory is that in any colouring of a sufficiently large structure we are guaranteed to find monochromatic copies of a given small structure. The classic example is Ramsey's theorem which concerns finding monochromatic complete graphs in colourings of large complete graphs. Specifically the Ramsey number $R(k,l)$ in the minimum $n$ for which every colouring of $E(K_n)$ with red and blue contains either a red $K_k$ or a blue $K_l$ (or equivalently every graph on $n$ vertices contains either a $K_k$ or a $\overline{K}_l$). Ramsey's Theorem is that $R(k,l)$ exists for all $k,l$. However, finding good bounds for Ramsey numbers is a notorious problem. See chapter 6 of \cite{MGT} for basic results in Ramsey theory and \cite{GRS} for a fuller account.

We make the following definition of the analogue of Ramsey numbers in the subcube intersection graph context:
\[
R_d(k,l)=\min\{n: \text{ every } G\in\I(n,d) \text{ contains either a $K_k$ or a } \overline{K}_l\}.
\]

Note that unlike the usual Ramsey number this definition is not symmetric in $k$ and $l$ (since $G\in\I(n,d)$ does not imply that $G^c\in\I(n,d)$). Hence the diagonal cases $R_d(k,k)$ are perhaps not the most natural. Rather we mainly consider the situation when $l$ is fixed and $k$ is arbitrary (even the case $l=3$ is interesting). This relates to the theorem of Berg et al on $(k,m)$-agreeable societies.
No graph in $\I(n,d)$ has an independent set of size larger than $2^d$. It follows that if $l>2^d$ we need only consider the restriction that $G$ is $K_k$-free and then trivially $R(k,l)=(k-1)2^d$. It makes sense therefore to assume that $l\leq 2^d$. 

Clearly $R_d(k,l)\leq R(k,l)$ and if $d$ is large we would expect this to be close to an equality. Indeed, any $d$ vertex graph can be expressed as the intersection graph of subcubes of $\{0,1\}^d$ (it is easy to prove this by induction on $d$ or see Tuza \cite{T} for a slightly stronger result). From this it follows that if $d\geq R(k,l)$ then $R_d(k,l)=R(k,l)$.

Taking $(k-1)$ copies of each of $(l-1)$ singletons we have the trivial lower bound $R_d(k,l)\geq(k-1)(l-1)+1$. This is clearly tight when $k=2$ and turns out to also be tight for $d=2$ (see appendix). For $d$ which are larger but not as large as $R(k,l)$ the situation is much more interesting. 

For $l=3$ we have the following result.

\begin{theorem}\label{Rk3:thm}
\[
d\lfloor\frac{k}{c\sqrt{d\log d}}\rfloor<R_d(k,3)<\frac{2d}{\log d-\log\log d}k
\]
For some absolute constant $c>0$.
\end{theorem}

For comparison the usual Ramsey number satisfies:

\begin{theorem}[(Kim and Ajtai-Koml\'os-Szemer\'edi)]
\[
c_1\frac{k^2}{\log k}<R(k,3)<c_2\frac{k^2}{\log k}
\]
For some absolute constants $c_1,c_2>0$.
\end{theorem}
with the lower bound due to Kim \cite{Kim} and the upper bound due to Ajtai, Koml\'os and Szemer\'edi \cite{AKS}.

Our lower bound is $0$ when $k<c\sqrt{d\log d}$ but in this range we are in the large $d$ situation described above when $R_d(k,3)=R(k,3)$. If $k\gg\sqrt{d\log d}$ then we obtain $R_d(k,l)>(1+o(1))c\sqrt{\frac{d}{\log d}}k$. As we mentioned above $R(k,l)$ is always an upper bound for $R_d(k,l)$ and if $d$ is much larger than $k$ then this will be smaller than the upper bound given.

For general $l$ the best bounds we have are:

\begin{theorem}\label{Rkl:thm}
\[
d\lfloor\frac{k}{c_ld^{\frac{2}{l+1}}\log d}\rfloor<R_d(k,l)< 2d^{l-2}k
\]
where $c_l>0$ depends on $l$ but not on $d$ and $k$.
\end{theorem}

Similarly to Theorem \ref{Rk3:thm} the lower bound is $0$ if $k<c_ld^{\frac{2}{l+1}}\log d$ while for $k\gg d^{\frac{2}{l+1}}\log d$ it is essentially $c_l\frac{d^{\frac{l-1}{l+1}}}{\log d}k$.

Obviously it would be interesting to improve these bounds and those of Theorem \ref{Rk3:thm}. We have no conjecture as to what the true values should be.

Interestingly, direct use of random methods does not seem to be helpful for constructing good lower bounds for $R_d(k,l)$. Instead the lower bounds in both cases involve graphs built from lower bounds for ordinary Ramsey numbers (of course these do typically involve random methods). We will use the fact mentioned above that any graph $G$ on $d$ vertices is the intersection graph of subcubes in $\{0,1\}^d$.

\begin{lemma}\label{Ramseylb:lem}
If $R(x,l)>d$ then $R_d(k,l)>d\lfloor\frac{k}{x}\rfloor$.
\end{lemma}

\begin{proof}
Since $R(x,l)>d$ there is a graph $G'$ on $d$ vertices with no $K_x$ and no $\overline{K}_l$. We can find subcubes $A_1,\dots,A_d$ in $Q_d$ whose intersection graph is $G'$. Now take $\lfloor\frac{k}{x}\rfloor$ copies of each $A_i$ and let $G$ be the intersection graph of these. Clearly $G\in\I(d\lfloor\frac{k}{x}\rfloor,d)$. The graph $G$ has no $\overline{K}_l$ since $G'$ has no $\overline{K}_l$. Every maximal clique in $G$ corresponds to taking all copies of $A_i$ for $i\in U$ where $U$ is the vertex set of a clique in $G'$. Hence $G$ contains no $K_{x\lfloor\frac{k}{x}\rfloor}$ and hence no $K_k$.
\end{proof}

This Lemma gives a simple to state relation between ordinary Ramsey lower bounds and lower bounds for $R_d(k,l)$. To improve readability in what follows we will generally content ourselves with using easy to work with lower bounds for $R(k,l)$ even if they are not quite the best known. However using stronger bounds in the Lemma (the strongest for fixed $l$ are due to Bohman and Keevash \cite{BK}) will in turn give stronger bounds for $R_d(k,l)$.

\begin{proof}[Proof of Theorem \ref{Rk3:thm}]
For the lower bound we use the fact that $R(x,3)>c_1 \frac{x^2}{\log x}$ and so $R(c_2\sqrt{d\log d},3)>d$ for a suitable constant $c_2$. It follows from Lemma \ref{Ramseylb:lem} that 
\[
R_d(k,3)>d\lfloor\frac{k}{c_3\sqrt{d\log d}}\rfloor.
\]

We now prove the upper bound. Suppose that we have a graph $G\in\I(n,d)$ which has no $K_k$ and no $\overline{K}_3$ and let $A_1,\dots,A_n$ be subcubes whose intersection graph is $G$. Suppose without loss of generality that $A_1\cap A_2=\emptyset$ and $\dim(A_1)=d-t$ where $t$ is the largest codimension of a subcube which is involved in a disjoint pair of subcubes. We will bound $n$ by considering subcubes in $V(G)$ according to whether or not they meet $A_1$.

Since $G$ does not contain a copy of $\overline{K}_3$ any subcube in $V(G)$ which does not intersect $A_1$ must intersect $A_2$. Moreover, any two subcubes in $V(G)$ which intersect $A_2$ but not $A_1$ are intersecting. Hence the set of subcubes in $V(G)$ which do not intersect $A_1$ forms a clique in $G$. It follows that there are at most $k-1$ of them.  

For notational convenience suppose that $F(A_1)=\{t+1,\dots,d\}$. Suppose that for every subcube $x$ which intersects $A_1$ we define the projection $\pi(x)$ to be the subcube of $\{0,1\}^t$ with $\pi(x)_i=x_i$ for $1\leq i\leq t$. Plainly if $\pi(A_i)\cap\pi(A_j)\not=\emptyset$ then $A_i\cap A_j\not=\emptyset$. Moreover since $A_i$ and $A_j$ both intersect $A_1$ we have that if $A_i\cap A_j\not=\emptyset$ then $\pi(A_i)\cap\pi(A_j)\not=\emptyset$. It follows that the intersection graph of those subcubes which intersect $A_1$ is the same as the intersection graph of their projections. Therefore the number of subcubes in $V(G)$ which intersect $A_1$ is at most $R_{d-t}(k,3)$. Putting these two bounds together we obtain:
\[
R_d(k,3)<k+R_{d-t}(k,3).
\]

A second bound comes from considering the sizes of the subcubes making up $G$. Under our assumptions, we have $m$ subcubes each with dimension at least $d-t$ and $n-m$ other subcubes each of which intersects every other subcube. Hence some point of $\{0,1\}^d$ is contained in at least $m2^{d-t}/2^d+(n-m)\geq n/2^t$ subcubes and so $n<2^tk$. Roughly speaking we will use this bound if $t$ is small and the inductive bound if $t$ is large. 

Specifically, fix $\alpha$ (we will choose $\alpha$ later to optimise the bound) and apply the inductive bound repeatedly working within a sequence of cubes $Q_d,Q_{d-t_0},Q_{d-t_0-t_1},\dots$ until the largest codimension of our subcubes (regarded as subcubes of $Q_{d-t_0-\dots-t_i}$) is less than $\alpha$. We now have that
\[
R_d(k,3)<ik+2^{\alpha}k.
\]
Since each step at which we applied the inductive bound involved moving to a cube of codimension at least $\alpha$ this can take at most $d/\alpha$ steps and so
\[
R_d(k,3)<\left(\frac{d}{\alpha}+2^{\alpha}\right)k.
\]
It remains to choose $\alpha$ to optimise this bound. There is no nice expression for such $\alpha$ but substituting $\alpha=\log d-\log\log d$ (which is certainly close to the minimum) gives
\[
R_d(k,3)<2\frac{d}{\log d-\log\log d}k.
\]
\end{proof}

A similar approach to the inductive bound gives the upper bound for general $R_d(k,l)$.

\begin{proof}[Proof of Theorem \ref{Rkl:thm}]
For the lower bound Lemma \ref{Ramseylb:lem} and the fact that $R(x,l)>c_l\left(\frac{x}{\log x}\right)^{\frac{l+1}{2}}$ (\cite{Spencer}) gives that $R(c_ld^{\frac{2}{l+1}}\log d,l)>d$. Hence $R_d(k,l)>c_l\frac{d^{\frac{l-1}{l+1}}}{\log d}k$.

For the upper bound we will use induction on $d$. The upper bound certainly holds for $d=1$. For larger $d$ take $G\in\I(n,d)$ with no $K_k$ and no $\overline{K}_l$. Suppose that $A_1,\dots,A_{l-1}$ are pairwise disjoint subcubes in $V(G)$ and define
\[
n_r=|\{x\in V(G): x\cap A_r\not=\emptyset, x\cap A_i=\emptyset \text{ for } 1\leq i<r\}|.
\]
The intersection graph induced by these $n_r$ subcubes is the same as the intersection graph of their projections onto $A_r$. Hence $n_r\leq R_{\dim(A_r)}(k,l-r+1)$. Every subcube meets one of the $A_1,\dots,A_{l-1}$ and so $n=\sum_{r=1}^{l-1} n_r$. So
\[
R(k,l)\leq\sum_{r=1}^{l-1}R_{\dim(A_r)}(k,l-r+1)\leq\sum_{r=1}^{l-1}R_{d-1}(k,l-r+1)
\]
($\dim(A_i)\not=d$ since $A_i$ is disjoint from some other subcubes.)

By the induction hypothesis
\[
R_d(k,l)<2k\sum_{i=1}^{l-1}(d-1)^{i-1}=2k\left(\frac{(d-1)^{l-1}-1}{d-2}\right)
\]
provided that $d>2$ (if $d=2$ the sum is $2k(l-1)\leq 2^{l-2}2k$).

It will suffice then to show that for $d>2$
\[
(d-1)^{l-1}-1\leq(d-2)d^{l-2}
\]
This is easily checked to be true (with equality) if $l=3$. If $l>3$ then the inequality follows from the fact that $(d-1)^3\leq (d-2)d^2$.
\end{proof}

\section{Further Questions}

We raised earlier the questions of determining more precisely the maximum number of edges in a $K_{r+1}$-free graph in $\I(n,d)$ and giving better bounds on the modified Ramsey numbers $R_d(k,l)$. In addition to these we have some related questions and directions for further study which we believe may be worthwhile although we do not even have preliminary results for them.

Firstly, as we saw in the upper bound for $R_d(k,3)$ (Theorem \ref{Rk3:thm}), knowing the dimensions of the subcubes involved constrains the intersection graph considerably. Motivated by this let $\I(n,d,r)$ be the set of all graphs on $n$ vertices which are the intersection graphs of $r$-dimensional subcubes in $\{0,1\}^d$. What can be said about the properties of graphs in $\I(n,d,r)$? Note that in the agreement model this is the situation where each individual has an opinion on exactly $d-r$ of the issues. It would also be natural (both mathematically and in the agreement model sense) to restrict the allowed dimensions to be in some range. For instance to have dimension at most $r$ or at least $r$. Both the Tur\'an and Ramsey type problems make sense with this extra parameter $r$. We could also asked for the smallest $r$ such that every $n$ vertex graph is in $\I(n,d,r)$ for some $d$. An answer to this would complement the result of Tuza \cite{T} mentioned earlier on the smallest $d$ such that every $n$ vertex graph is in $\I(n,d)$.

Also motivated by the agreement model it may be interesting to weaken slightly the notion of intersection to some measure of closeness. In $\{0,1\}^d$ the appropriate notion for this is Hamming distance which for $x,y\in\{0,1\}^d$ is defined by $d_H(x,y)=|\{i:x_i\not= y_i\}$. A sample question then would be:

\begin{q}
Given $A_1,\dots,A_n$ subcubes in $\{0,1\}^d$ and an integer $t$ what is the minimum number of edges in the corresponding intersection graph which guarantees there is a point $x\in\{0,1\}^d$ which is within Hamming distance $t$ of at least $r$ of the $A_i$?
\end{q}

Regarding our subcube intersection graphs as discrete analogues of $d$-box graphs suggests allowing larger discrete boxes. Specifically, let $X=[k]^d$ and let each $A_i$ be the product of intervals of integers in $[k]$. What can be said about this family of interval graphs both for fixed $k$ and allowing $k$ to grow?

As we mentioned earlier the boxicity of a graph $G$ is the minimum dimension $d$ for which $G$ is the intersection graph of a family of axis-parallel boxes in $\mathbb{R}^d$. We may define the \emph{binary boxicity} of $G$ to be the minimum dimension $d$ for which $G$ is the intersection graph of a family of subcubes in $Q_d$. This is not a new graph parameter; it is the binary clique cover number of the complement of $G$. However, the idea of binary boxicity may be a new perspective on it. The usual boxicity of a graph has been used to measure the complexity of graphs in various contexts such as the theory of ecological and social networks. It may be that the binary boxicity could have similar uses where the situation is naturally discrete.

Finally, random subcube intersection graphs could be studied. This complements a model of random intersection graphs, introduced in \cite{KSS} which has been the subject of a number of papers. The natural model for us is to choose some $0<p<1/2$ and select a random vector in $C_d=\{0,1,*\}^d$ by setting each coordinate to be $0$ with probability $p$, $1$ with probability $p$, and $*$ with probability $1-2p$ independently of the other coordinates. Repeating this to choose $n$ independent random elements of $C_d$ gives a random intersection graph in $\I(n,d)$. There is considerable dependence between the events that particular edges are present which may make this model challenging to analyse. For large $d$ this model will with high probability give a set of subcubes with nearly equal dimensions. An interesting variation would be to generate each subcube by picking its codimension from a distribution on $[d]$ which is not highly concentrated, and then picking a fixed set uniformly among the sets of that size.

\section*{Appendix: Ramsey results for small $d$}

Although our main focus is approximate results for arbitrary large $d$ we give here a few results for $d=2,3$. When $d=2$ it is easy to determine $R_d(k,l)$ exactly

\begin{theorem}
For $l=2,3,4$ and all $k$ we have $R_2(k,l)=(k-1)(l-1)+1$.
\end{theorem}

\begin{proof}
It is trivial that $R_2(k,l)\geq(k-1)(l-1)+1$ and that $R_2(k,2)=k$ (in fact $R_d(k,2)=k$ for all $d$).

For $R_2(k,3)$ suppose that we have a multiset $V$ of $n$ subcubes whose intersection graph does not contain $K_k$ or $\overline{K}_3$. We will show that the largest possible such $n$ is $2(k-1)$. If $V$ contains two distinct singleton subcubes then every subcube in $V$ meets at least one of these singletons (since the intersection graph in $\overline{K}_3$-free). Each singleton is contained in at most $(k-1)$ subcubes in $V$ and so $n\leq 2(k-1)$. If we do not then we may assume without loss of generality than $V$ does not contain either of the singleton subcubes $(0,0)$ and $(1,1)$. But now every subcube in $V$ meets either $(0,1)$ or $(1,0)$ and as before each of these is contained in at most $(k-1)$ subcubes in $V$. It follows that $n\leq 2(k-1)$.

Finally for $R_2(k,4)$ suppose that we have a multiset $V$ of $n$ subcubes whose intersection graph does not contain $K_k$ or $\overline{K}_4$. We cannot have 4 distinct singleton subcubes in $V$ and so we may assume without loss of generality that $V$ does not contain $(0,0)$. But now every subcube in $V$ must meet at least one of the points $(0,1),(1,0),(1,1)$. Since each of these points contains at most $(k-1)$ subcubes in $V$ we conclude that $n\leq 3(k-1)$.
\end{proof}

This kind of case by case argument becomes rather messy even for $d=3$. However in this case we are small enough that an exhaustive computer search is possible. The exact results in the table below are the outcome of this search. In particular note that $R_3(k,l)\not=(k-1)(l-1)+1$ (so long as $k,l\geq3$) so even in this small dimension the situation is not trivial. We also have that $R_3(k,l)\not=R(k,l)$ unless $k=l=3$. For each $(k,l)$ considered we give a graph on $R_3(k,l)-1$ vertices which does not contain a $K_k$ or a $\overline{K}_l$. The vertices of this graph are given as elements of $C_3$.

\[
\left.
\begin{array}{|c|c|l|}
\hline
(k,l) & R_3(k,l) & \text{Extremal graph}\\
\hline\hline
(3,3) & 6 & (**0),(**0),(*11),(0*1),(11*)\\ \hline
(4,3) & 8 & (**0),(**0),(**1),(*0*),(*11),(0*1),(11*)\\ \hline
(5,3) & 11 & (**0),(**0),(*0*),(*0*),(*11),(*11),(0*1),(0*1),(11*),(110)\\ \hline
(6,3) & 13 & (**0),(**0),(**0),(**1),(*0*),(*0*),(*11),(*11),(0*1)\\ 
& &(0*1),(11*),(11*)\\ \hline
(3,4) & 8 & (**0),(*01),(0*1),(01*),(10*),(11*),(111)\\ \hline
(4,4) & 11 & (**0),(**0),(*01),(*01),(*11),(0*1),(01*),(10*),(11*),(111)\\
\hline
\end{array}
\right.
\]


\end{document}